\documentclass[12pt]{elsarticle}

\usepackage{amssymb,amsthm, url,amsmath,graphicx}
\usepackage{tikz}
\usepackage{pgfplots}
\usepackage{subcaption}

\newcommand{\MCM}{\mathcal{M}}
\newcommand{\MCV}{\mathcal{V}}
\newcommand{\mcA}{\mathcal{A}}

\newtheorem{theorem}{Theorem}[section]
\newtheorem{proposition}{Proposition}[section]
\newtheorem{lemma}{Lemma}[section]

\pgfplotsset{compat=1.13}

\usepackage{lineno}

\journal{Computers and Mathematics with Applications}

\begin{document}

\begin{frontmatter}


\title{Optimal-order preconditioners for the Morse-Ingard equations\tnoteref{t1}}

\tnotetext[t1]{This work was supported by NSF 1620222.}
\author[bu]{Robert C. Kirby\corref{cor1}}
\ead{robert\_kirby@baylor.edu}

\author[bu]{Peter Coogan}
\ead{peter\_coogan@baylor.edu}

\cortext[cor1]{Corresponding author}
\address[bu]{Department of Mathematics, Baylor University, One Bear Place \#97328, Waco, TX 76798-7328, United States}


\begin{abstract}
  The Morse-Ingard equations of thermoacoustics~\cite{acoustics} are a
  system of coupled time-harmonic equations for the temperature and
  pressure of an excited gas.  They form a critical aspect of modeling
  trace gas sensors.  In this paper, we analyze a reformulation of the
  system that has a weaker coupling between the equations than the
  original form.  We give a G{\aa}rding-type inequality for the system
  that leads to optimal-order asymptotic finite element error
  estimates.  We also develop preconditioners for the coupled
  system.  These are derived by writing the system as a $2 \times 2$
  block system with pressure and temperature unknowns segregated into
  separate blocks and then using either the block diagonal or block lower
  triangular part of this matrix as a preconditioner.  Consequently,
  the preconditioner requires inverting smaller, Helmholtz-like
  systems individually for the pressure and temperature.  Rigorous
  eigenvalue bounds are given for the preconditioned system, and these
  are supported by numerical experiments.
\end{abstract}


\end{frontmatter}


\section{Introduction}
\label{sec:intro}
Laser absorption spectroscopy is used for detecting trace amounts of
gases, with applications to diverse fields such as air quality
monitoring, disease diagnosis, and manufacturing~\cite{curl2010quantum,patimisco2014quartz,petersen2017quartz}.  One
particular approach is photoacoustic spectroscopy, in which a laser is
fired between the tines of a small quartz tuning fork.  With appropriate
configuration of the laser and geometry, even minute amounts of the
gas can generate acoustic and thermal waves.  The interactions of
these waves with the tuning fork induces an electric signal due to
pyroelectric and piezoelectric effects.

Accurate computational modeling of these sensors provides an
important first step in optimizing the design of sensors to maximize
sensitivity subject to manufacturing constraints.
Two variants of these sensors are the so-called QEPAS (quartz-enhanced
photoacoustic spectroscopy) and ROTADE (resonant optothermoacoustic
detection) models~\cite{kosterev2002quartz, kosterev2010resonant}.  In
QEPAS, the acoustic wave dominates the signal, while the thermal wave
is more important in ROTADE. In many experimental
configurations, both effects appear.

Earlier work on modeling this problem~\cite{petraThesis,
  QEPASdesign, ROTADEdesign} simplified the model to a single PDE.
This includes an empirically-determined damping term to account for
otherwise-neglected processes and only works in select parameter regimes
Moreover, the empirical terms contain parameters
that strongly depend on particular geometry and not just physical
material parameters.  This greatly complicates computationally optimizing over
geometry.  Wanting to bypass this constraint,
a finite element discretization of the coupled 
pressure-temperature system was first addressed in~\cite{pyHPC}, where the
difficulty of solving the linear system was noted.  Kirby and Brennan
gave a more rigorous treatment in~\cite{brennan2015finite}.  This included error
estimates and the introduction and analysis of block preconditioners.
Kaderli \emph{et al} derived an analytical solution for the coupled
system in idealized geometry in~\cite{kaderli2017analytic}.  Their technique
involves reformulating the system studied in~\cite{brennan2015finite} by an algebraic
simplification that eliminates the temperature Laplacian from the
pressure equation.  Recent work by Safin \emph{et al}~\cite{safin} made
several concrete advances.  For one, they coupled the Morse-Ingard equations for
atmospheric pressure and temperature to heat conduction of the
quartz tuning fork, although vibrational effects were still not considered.
A perfectly-matched layer (PML)~\cite{berenger1994perfectly} can also be
used to truncate the computational domain, and a Schwarz-type
preconditioner that separates out the PML region was used to
effectively reduce the cost of solving the linear system.
They also include some favorable comparisons between the computational
model and experimental data. 

In this paper, we study further aspects of this
pressure/temperature system.  We extend the analysis of~\cite{brennan2015finite} to the
reformulated system of~\cite{kaderli2017analytic}.  Although the new system is no longer
coercive, it does admit a G{\aa}rding inequality that leads to finite
element error estimates.  We also give a rigorous
treatment of the eigenvalue clustering for block preconditioners for this formulation.
Compared to those studied in~\cite{brennan2015finite}, we obtain mesh-independent
results that in practice give much lower iteration counts than for the
original system.  We have not addressed PML or
other more robust absorbing boundary conditions, but we do give
rigorous analysis of the pressure-temperature system.
While~\cite{safin} considers this pressure temperature model under
PML, the theoretical results there focus on a domain decomposition
technique to separate the boundary region from the interior for the
simpler Helmholtz model.

The rest of the paper is organized as follows.  In
Section~\ref{sec:MI}, we present the Morse-Ingard equations, their
reformulation, and the resulting finite element discretization.  After
recalling some finite element convergence theory for Helmholtz-type
operators in Section~\ref{sec:fea}, we also discuss the connection
between linear algebra and operators on the discrete spaces and give
finite element error estimates for the reformulated Morse-Ingard system.  A
description of our block preconditioners together with theoretical
investigation then follows in Section~\ref{sec:bp}, after which we
present numerical results in Section~\ref{sec:num}.

\section{The Morse-Ingard equations and finite element discretization}
\label{sec:MI}

The Morse-Ingard equations~\cite{acoustics} are posed for pressure \( P \) and
temperature \( T \) in a bounded domain \( \Omega \subset \mathbb{R}^d \) for
\( d = 2 \) or \( 3 \).  The pressure and temperature satisfy a wave
and heat equation respectively, although they are coupled via viscous
forces.  The equations are
\begin{equation}
\begin{split}
\tfrac{\partial}{\partial t}
\left( T - \tfrac{\gamma - 1}{\gamma \alpha} P\right)
- \ell_h c \Delta T & = S, \\
\gamma \left( \tfrac{\partial^2}{\partial t^2} - \ell_v
c \tfrac{\partial}{\partial t} \Delta\right)
\left( P - \alpha T \right) - c^2 \Delta P & = 0,
\end{split}
\end{equation}
where $\ell_v$ and $\ell_h$ are characteristic
lengths associated with the respective effects of fluid viscosity and thermal conduction, $c$ is sound speed, $\gamma$ is the ratio of the specific heat
of the gas at constant pressure to that at constant volume, $\alpha = \left(\tfrac{\partial P}{\partial T} \right)_v$ is the rate of
change of ambient pressure with respect to ambient temperature at
constant volume, and $\omega$ is the frequency of a forcing function
applied to the system.

This system models fairly general thermoacoustic waves propagating
in a fluid.  It is assumed that, absent the waves, the fluid
is at rest.  Hence, the derivativation begins by linearizing the
incompressible Navier-Stokes equations around zero and including
acoustic and thermal effects.  For a derivation and particular
physical assumptions, we refer the reader to~\cite{acoustics,kaderli2017analytic}.  We will also
assume a periodic forcing function $S$ to reduce to a time-harmonic
system of equations.  From an engineering perspective, this is not a
major restriction.  This model does not account for
interactions of the gas with solid boundaries, which can occur either through heat exchange or
fluid-structure interaction.  However, such models will still include the system we consider in this paper as a sub-problem.

Following~\cite{kaderli2017analytic}, we will
also introduce the parameters
\begin{equation}
  \MCM := \tfrac{\ell_h \omega}{c}, \ \ \ 
  \Lambda := \tfrac{\ell_v \omega}{c}.
 \label{eq:mcmlam}
\end{equation}
In Table~\ref{table:params}, we show representative parameters for the
QEPAS and ROTADE sensors, taken from~\cite{safin}.

\begin{table}
  \centering
  \begin{tabular}{|l|c|c|}\hline
    & QEPAS & ROTADE \\ \hline
    $\ell_v$ & $1.537\times 10^{-7} \text{m}$
    & $1.383\times 10^{-5}\text{m}$ \\ 
    $\ell_h$ & $1.0157 \times 10^{-7} \text{m}$ &
    $9.144 \times 10^{-6} \text{m}$ \\ 
    $\alpha$ & $204.656 \tfrac{\text{Pa}}{\text{K}}$ &
    $2.274 \tfrac{\text{Pa}}{\text{K}}$ \\ 
    $c$ & $348.7 \tfrac{\text{m}}{\text{s}}$ & $348.7
    \tfrac{\text{m}}{\text{s}}$ \\ 
    $\gamma$ & $\tfrac{7}{5}$ & $\tfrac{7}{5}$ \\ 
    $\omega$ & $2.061 \times 10^5 \tfrac{\text{rad}}{\text{s}}$ &
    $2.061 \times 10^5 \tfrac{\text{rad}}{\text{s}}$ \\
$\MCM$ & $6.003 \times 10^{-5}$ & $5.404 \times 10^{-3}$ \\
    $\Lambda$ & $9.084 \times 10^{-5}$ & $8.179\times 10^{-3}$ \\ \hline
  \end{tabular}
  \caption{Representative values for the physical parameters in the
    QEPAS and ROTADE regimes.  Note that $c$, $\gamma$, and $\omega$
    are the same in both cases.}
    \label{table:params}
\end{table}

In our applications, $\Omega$ is typically the exterior of a tuning
fork, truncated a sufficient distance away.  A simple mesh of a
two-dimensional domain is shown later in Figure~\ref{fig:tfmesh}.
In light of the physical
discussion, we partition the boundary into \(
\partial \Omega = \Gamma_a \cup \Gamma_w \).  Here,  \( \Gamma_a \)
denotes the ``air'' boundary at which we truncate the computational
domain from the outside world.  \( \Gamma_w \) denotes the ``wall''
boundary separating the air from the quartz tuning fork.  On \(
\Gamma_a \), we use the boundary conditions
\begin{equation}
  \label{eq:gammaabc}
  \begin{split}
  \tfrac{\partial T}{\partial n} & = 0, \\
  \tfrac{\partial P}{\partial n} - i\tfrac{\sqrt{\gamma}\omega}{c}P & =
  0.
  \end{split}
\end{equation}
The first of these boundary conditions corresponds to the assumption
that all heat dissipation occurs near the tuning fork, and the second
is the standard ``transmission'' boundary condition.  Although better
domain truncation may be obtained by means of PML~\cite{berenger1994perfectly} or a
nonlocal condition~\cite{johnson1980coupling}, solvers under such
conditions frequently require handling transmission boundary
conditions as an important sub-problem.
We hope to study more advanced boundary conditions for the Morse-Ingard equations in future work.

On the wall boundary \( \Gamma_w \), we pose the no-flux conditions
\begin{equation}
  \label{eq:gammawbc}
  \begin{split}
    \tfrac{\partial T}{\partial n} & = 0, \\
    \tfrac{\partial P}{\partial n} & = 0.
  \end{split}
\end{equation}
The first of these conditions corresponds to thermally insulating the
air from the tuning fork.  Coupling between the air
and tuning fork requires generalizing this condition and is considered
in~\cite{safin}.  The second
condition states that acoustic waves will reflect off of the tuning
fork.  Handling the interaction between acoustic waves and the
vibrations of the tuning fork requires more complicated physics and
will be the subject of future investigation. 

Because the forcing function \( S \) is time-harmonic with frequency
\( \omega \) in our applications, the linearity of the equations means
that we can consider the time-harmonic form of this system:
\begin{equation}
\label{eq:harm}
\begin{split}
-\ell_h c \Delta T -i \omega T + i\omega \tfrac{\gamma - 1}{\gamma \alpha}P & = S,
 \\
- i \gamma \ell_v c \omega \alpha \Delta T + \gamma \omega^2 \alpha T
- (c^2 - i \gamma \ell_v c \omega) \Delta P
-\gamma \omega^2 P
& = 0.
\end{split}
\end{equation}

In addition to the parameters introduced in~\eqref{eq:mcmlam}, we
follow~\cite{kaderli2017analytic} in defining the nondimensional variables
\begin{gather*}
    \mathbf{x}_* = \tfrac{\omega\mathbf{x}}{c}, 
    \nabla_*  = \tfrac{c}{w} \nabla, \\
    P_*  = P\left( \tfrac{c \mathbf{x}_*}{\omega} \right), \\
    T_*  = \alpha T\left( \tfrac{c \mathbf{x}_*}{\omega} \right), \\
    S_*  = -\tfrac{\alpha}{\omega} S\left( \tfrac{c \mathbf{x}_*}{\omega} \right),
  \label{eq:newvars}
\end{gather*}
and then dropping the stars, the system becomes
\begin{equation}
  \label{eq:dropstar}
  \begin{split}
    \MCM \Delta T + i T - i \tfrac{\gamma-1}{\gamma} P & = S, \\
    \Delta P + \gamma \left( 1 - i \Lambda \Delta \right) \left( P - T
    \right) & = 0.
  \end{split}
\end{equation}
With $\MCM$ small and other coefficients $\mathcal{O}(1)$, the first
equation amounts to a large but skew perturbation of the Laplacian for
temperature together with a large but 0-order coupling with pressure.
The second equation has an indefinite Helmholtz operator acting on
pressure, although the effective wave number is moderate.  However,
the temperature appears in this equation both with a moderate 0-order term as
well as with (small parameter) times the Laplacian.

This situation can be favorably altered
following~\cite{kaderli2017analytic}.  We can subtract \(
i\gamma\tfrac{\Lambda}{\MCM} \) times the first equation from the
second to eliminate the temperature Laplacian.  We also negate both
equations so that the negative rather than positive Laplacian appears.
\begin{equation}
  \label{eq:PDE}
  \begin{split}
  -\MCM \Delta T - i T + i \tfrac{\gamma-1}{\gamma} P & = -S, \\
   \gamma\left( 1 - \tfrac{\Lambda}{\MCM} \right) T
  -\left( 1 - i \gamma \Lambda \right) \Delta P
  - \left[ \gamma \left(1-\tfrac{\Lambda}{\MCM}\right)-
    \tfrac{\Lambda}{\MCM}\right]P
  & = i\gamma\tfrac{\Lambda}{\MCM}S.
  \end{split}
\end{equation}

We also need to reformulate the transmission boundary condition
on~\eqref{eq:gammaabc} to account for the nondimensionalization.  In
particular, using the substitution~\eqref{eq:newvars} gives that
\begin{equation}
  \tfrac{\partial P}{\partial n} - i \sqrt{\gamma} P = 0.
\end{equation}

This reformulation leaves only a single Laplacian in each equation, so
that the coupling is only through zero-order terms.   We observe that the second
equation now has an indefinite Helmholtz operator on pressure, but for
our parameters, the wave number $\kappa^2 \sim 2.23$.   The coupling
to temperature has a similarly moderate size.

In~\cite{kaderli2017analytic}, this reformulation
enabled an analytical solution in the case of cylindrical symmetry
with Gaussian forcing, but we give a
finite element analysis amenable to less-idealized configurations.
In particular, we find that this reformulated system also
leads to more efficient preconditioners than we considered for the
original system.

We proceed by setting further notation.  Let \( L^2(\Omega) \) denote
the standard space of square-integrable complex-valued functions over
\( \Omega \) and \( H^k(\Omega) \subset L^2(\Omega) \) the space
consisting of functions with square-integrable weak derivatives of
order up to and including \( k \geq 0 \).  We let \( \| \cdot \|_\MCV \)
denote the norm associated with any space \( \MCV \), omitting it when
\( \MCV =L^2(\Omega) \).  We frequently omit the argument so that
from $L^2 \equiv L^2(\Omega)$ and
$H^k \equiv H^k(\Omega)$. The $H^k$ seminorm $|\cdot|_{H^k}$ will consist of $L^2$ norms
of all partial derivatives of order exactly $k$ in the standard way.

The space \( L^2(\Omega) \) is equipped with the standard
inner product
\begin{equation}
  (f, g) = \int_\Omega f(x) \overline{g(x)} dx,
\end{equation}
and we also have the inner product over any portion of the boundary \(
\Gamma \subseteq \partial \Omega \)
\begin{equation}
  \langle f, g \rangle_{\Gamma} = \int_\Gamma f(s) \overline{g(s)} ds.
\end{equation}

Since we are dealing with a system of two PDE, we will also need norms on the
Cartesian product of spaces.  To that end, for any \( U = (u,v) \in H^s(\Omega) \times H^s(\Omega) \), we write
\begin{equation}
  \| U \|^2_{H^s} = \| (u, v) \|_{H^s}^2 = \| u \|^2_{H^s} + \| v \|^2_{H^s}.
\end{equation}

We arrive at a weak form of the system~\eqref{eq:PDE} subject to the given boundary conditions by choosing some test functions \( v, w \in H^1(\Omega) \), multiplying the first
equation by \( \overline{v} \) and the second by \( \overline{w} \)
and integrating each over \( \Omega \).  Applying the boundary conditions after integration by parts gives
\begin{equation}
  \label{eq:weakform}
  \begin{split}
    \MCM \left( \nabla T, \nabla v \right) - i \left( T, v \right)
    + i \tfrac{\gamma-1}{\gamma} \left( P, v \right) & =
    -\left(S, v \right), \\
   \gamma\left( 1 - \tfrac{\Lambda}{\MCM} \right) \left( T, w \right)
   + \left( 1 - i \gamma \Lambda \right)
   \left[ \left(\nabla P, \nabla w\right)
     - i \sqrt{\gamma} \langle P , w \rangle_{\Gamma_a} \right] \\
  - \left[ \gamma \left(1-\tfrac{\Lambda}{\MCM}\right)-
     \tfrac{\Lambda}{\MCM}\right]\left(P, w\right)
  & = i\gamma\tfrac{\Lambda}{\MCM}\left(S, w\right)
  \end{split}
\end{equation}
Equivalently, writing $U=(T,P) \in H^1\times H^1$ and adding the
equations together, we may write this as
\begin{equation}
  \label{eq:abstweakform}
  a(U, V) = F(V)
\end{equation}
for all \( V=(v, w) \in H^1 \times H^1 \), where
\begin{equation}
  \begin{split}
    a(U, V) = & \MCM \left( \nabla T, \nabla v \right) - i \left( T, v \right)
    + i \tfrac{\gamma-1}{\gamma} \left( P, v \right) \\
    & +\gamma\left( 1 - \tfrac{\Lambda}{\MCM} \right) \left( T, w \right)
    + \left( 1 - i \gamma \Lambda \right)
    \left[ \left(\nabla P, \nabla w\right)
      - i \sqrt{\gamma} \langle P , w \rangle_{\Gamma_a} \right] \\
    & - \left[ \gamma \left(1-\tfrac{\Lambda}{\MCM}\right)-
       \tfrac{\Lambda}{\MCM}\right]\left(P, w\right) \\
    F(V) = &  -(S, v) + i \gamma \tfrac{\Lambda}{\MCM} \left(S, w  \right).
  \end{split}
  \label{eq:aandF}
\end{equation}

We will assume that for any \( F \) in the dual of \( H^1(\Omega)
\times H^1(\Omega) \) that the variational problem and its adjoint each have
a unique solution.  For physical parameters of particular interest,
this has been proven~\cite{brennan2015finite}.  We also assume
regularity in that there exists a constant \( C_R \) such that if \( F \in L^2
\times L^2 \), then the corresponding solution \( U = (T, P) \) is in \( H^2 \times H^2 \) and
\begin{equation}
  \label{eq:regassumption}
  \| U \|_{H^2} \leq C_R \| F \|_{L^2}.
\end{equation}

We partition the domain $\Omega$ into conforming,
quasiuniform triangulations~\cite[Chapter~3]{BrennerScott} and let $\MCV_h$ consist of
continuous piecewise polynomials of some degree $k$, typically 1. 
For all finite element spaces we use, the standard approximation property
\begin{equation}
  \label{eq:approxassumption}
  \inf_{v \in V_h} \| u - v \|_{H^1} \leq C_A h |u|_{H^2}
\end{equation}
holds for any \( u \in H^1 \), as well as the inverse inequality
\begin{equation}
  \label{eq:inverse}
  \| u \|_{H^1} \leq \tfrac{C_I}{h} \| u \|_{L^2}, \ \ \ u \in \MCV_h.
\end{equation}
We also introduce $\MCV^2_h = \MCV_h \times \MCV_h$ consisting of
pairs of finite element functions.  We use the standard Galerkin
discretization of our system, seeking a
numerical solution by restricting the bilinear form to the
finite-dimensional subspace $\MCV^2_h \times \MCV^2_h$.  In this case,
we introduce the discrete approximation of finding $U_h \in \MCV^2_h$
such that
\begin{equation}
  \label{eq:discabstweakform}
  a(U_h, V_h) = F(V_h)
\end{equation}
for all $V_h \in \MCV^2_h$, where $a$ and $F$ are the same forms
in~\eqref{eq:aandF}.

\section{Finite element analysis}
\label{sec:fea}
\subsection{Helmholtz-type equations}
Before proceeding to the Morse-Ingard equations, we recall several
facts about finite element discretization of Helmholtz-type
equations.  Equations of the form $-\Delta u - \kappa^2 u = f$ are
elliptic but typically not coercive for even moderate $\kappa$.  Hence, the standard theory based
on the Lax-Milgram and C\'ea Lemmas must be extended.
Following~\cite{BrennerScott} and the references therein, we can
establish finite element solvability and error estimates for a general
bounded bilinear form $a(u,v)$ on $H^1 \times H^1$ provided several conditions
are met.  First, one posits that the underlying equation is uniquely
solvable (e.g. the Laplacian is not shifted by an eigenvalue) and has
a regularity estimate of the form $\| u \|_{H^2} \leq C_R \| f
\|_{L^2}$.
Second, the bilinear form is bounded, so that
\begin{equation}
a(u,v) \leq C_1 \| u \|_{H^1} \| v \|_{H^1}, \ \ \ u, v \in H^1.
\end{equation}

Third, the bilinear form must satisfy a G{\aa}rding-type inequality.
That is, there must exist some $\alpha > 0$ and $K$ such that the shifted bilinear form
is coercive, with
\begin{equation}
  \label{eq:garding}
a(u, u) + K \| u \|^2 \geq \alpha \| u \|_{H^1}^2.
\end{equation}

Note that if~\eqref{eq:garding} holds for $K\leq 0$, then the bilinear
form is in fact coercive.  Even if $K > 0$ is required, it is possible
to prove that the finite element approximation of $u_h \in V_h$ such that
\begin{equation}
a(u_h, v_h) = (f, v_h), \ \ \ v_h \in V_h
\end{equation}
is well-posed and satisfies optimal-order error estimates provided that the mesh is sufficiently fine.  In particular, one must have $h \leq h_0$ with
\begin{equation}
h_0 = \tfrac{\sqrt{\alpha}}{C_1C_AC_R \sqrt{2K}},
\end{equation}
but no restriction is required if~\eqref{eq:garding} holds for $K \leq
0$.  Supposing any required condition on $h_0$ holds,
$u_h$ uniquely exists and 
\begin{equation}
\| u - u_h \|_{H^1} \leq C \inf_{v \in \MCV_h} \| u - v \|_{H^1},
\end{equation}
where $C = \tfrac{2C_1}{\alpha}$.  One also obtains optimal-order $L^2$
estimates
\begin{equation}
  \| u - u_h \| \leq C_1 C_A C_R h \| u - u_h \|_{H^1}.
\end{equation}

We will build from this theory in two ways.  First, the abstraction carries over straightforwardly to systems and complex-valued problems.  In Section \ref{sec:femtheory}, we demonstrate a G{\aa}rding-type inequality for the Morse-Ingard bilinear form and thence state error estimates.
Second, the uniform solvability for $h \leq h_0$ means that one has a mesh-independent bound on the inverse operators restricted to finite element spaces.  
Our analysis of preconditioners in Subsection~\ref{subsec:bp} will
utilize this uniform solvability for Helmholtz-type problems to give
estimates for certain products of finite element matrices.

As in~\cite{kirby2010functional}, it is natural to think of a bilinear
form $a$ as 
encoding a discrete 
operator $\mcA_h: V_h \rightarrow V_h^\prime$ satisfying
$\langle \mcA_h u_h, v_h \rangle = \langle f , v_h \rangle$, where
$\langle \cdot , \cdot \rangle$ is the $H^1$ duality pairing.  The
boundedness of the bilinear form $a$ proves that $\mcA_h$ is a bounded
operator into the dual with norm uniformly bounded in $h$.  Moreover, the error
estimates also show that the norm can be used to show
that $\mcA_h$ has an inverse uniformly bounded in $h$ as well.  To
wit, let $f \in (L^2)^\prime$ and $u_h$ solve $\mcA_h u_h = f$.  Then
\begin{equation}
  \label{eq:unifinv}
\| \mcA_h^{-1} f \|_{H^1} = \| u_h \|_{H^1} 
\leq \| u \|_{H^1} + \| u - u_h \|_{H^1}
\leq C_R(1 + C C_A h) \| f \|, 
\end{equation}
and the constant is bounded above since $h \leq h_0$.

Let $\{\psi_{\imath}\}_{\imath=1}^{\dim \MCV_h}$ be a basis for
the finite element space $V_h$.  (Note: we use $\imath$ and $\jmath$
instead of $i$ and $j$ to prevent confusion with the complex unit).  Then, we define
\begin{equation}
  \label{eq:massandstiff}
  A_{\imath\jmath} = a(\psi_{\jmath}, \psi_{\imath}), \ \ \ M_{\imath\jmath} = (\psi_{\jmath}, \psi_{\imath})
\end{equation}
to be the stiffness matrix arising from Galerkin discretization and the mass 
matrix, respectively.   We can identify any $u, v \in
\MCV_h$ with their respective vectors of expansion coefficients $\mathrm{u},
\mathrm{v}$.

Recalling the discussion in~\cite{kirby2010functional}, we note that the mass matrix $M$ plays an important role connecting the $L^2$ inner product and norm on $\MCV_h$ to linear algebra.  The $L^2$ inner product on $\MCV_h$ is just realized as the $M$-inner product on $\mathbb{C}^n$:
\begin{equation}
  (u, v) = \mathrm{v}^* M \mathrm{u},
\end{equation}
where the order of the arguments gets the complex conjugate in the
right place.  Similarly, the $L^2$ norm is related to the $M$-induced
vector norm by
\begin{equation}
\| u \|^2 = \mathrm{u}^* M \mathrm{u} = \| \mathrm{u} \|_M^2.
\end{equation}

Suppose a matrix $\tilde{A}$ encodes a bounded linear map $\tilde{\mcA}_h$ from $\MCV_h$ to itself, with input and result expressed in the same basis $\{ \psi_{\imath} \}_{\imath=1}^{\dim \MCV_h}$.  We can relate the matrix norm induced by the $M$-norm to the canonical operator norm for bounded maps on $L^2$ by
\begin{equation}
  \label{eq:opnorm}
\| \tilde{A} \|_M = \sup_{\| \mathrm{u} \|_M = 1} \| \tilde{A} \mathrm{u} \|_M
= \sup_{\| u \| = 1} \| \tilde{\mcA}_h u \|
= \| \tilde{\mcA}_h \|.
\end{equation}

The mass matrix also encodes the $L^2$ Riesz map for functions in $\MCV_h$ and its $L^2$ dual.  Let $\tau: \MCV_h \rightarrow \MCV_h^\prime$ identify a member of $\MCV_h$ with a linear function by means of
\begin{equation}
  \langle \tau u , v \rangle = (u, v) = \mathrm{v}^*M\mathrm{u}.
\end{equation}
The Riesz Representation Theorem states that $\tau$ is an isometric isomorphism -- any $f \in \MCV_h^\prime$ has a unique $\tau^{-1} f \in \MCV_h$ so that for all $u \in \MCV_h$,
\begin{equation}
  (\tau^{-1} f, u) = \langle f, u \rangle.
\end{equation}
If $\mathrm{f}$ encodes coefficients of some $f \in \MCV_h^\prime$ in the basis dual to $\{\psi_{\imath}\}_{\imath=1}^{\dim V_h}$, then $M^{-1} \mathrm{f}$ gives the coefficients of $\tau^{-1} f$ relative to $\{\psi_{\imath}\}_{\imath=1}^{\dim V_h}$ .

Returning to the stiffness matrix, we have that
\begin{equation}
\mathrm{v}^* A \mathrm{u} = a(u,v) = \langle \mcA_h u, v \rangle.
\end{equation}
Since $A$ encodes $\mcA_h: \MCV_h \rightarrow \MCV_h^\prime$, $M^{-1} A$
encodes the operator $\tau^{-1} \mcA_h : \MCV_h \rightarrow \MCV_h$
and $A^{-1} M$ encodes its inverse $(\mcA_h)^{-1} \tau : \MCV_h
\rightarrow \MCV_h$. 

Equation~\eqref{eq:opnorm} provides a starting point to bound $M$ norms for the matrix $M^{-1} A$ and its inverse.  Since $M^{-1}A$ encodes $\tau^{-1} \mcA_h$, we have that
\begin{equation}
  \label{eq:minvam}
\| M^{-1} A \|_M = \| \tau^{-1} \mcA_h \| = \| \mcA_h \|.
\end{equation}

These are $L^2$-based norms on $\MCV_h$ and operators and not
$H^1$ norms.  Using~\eqref{eq:opnorm} together with the inverse estimate~\eqref{eq:inverse}, we have:
\begin{equation}
  \begin{split}
  \| \mcA_h \| = \sup_{\| u \| = \| v \| = 1} |a(u,v)|
  \leq \sup_{\| u \| = \| v \| = 1} C_1 \| u \|_{H^1} \| v \|_{H^1} 
  \leq C_1 C_I^2 h^{-2}
  \end{split}
\end{equation}

While the $L^2$ norm of the operator grows like $h^{-2}$, which is
expected for discretizing a second-order elliptic operator, the norm
of $\mcA_h^{-1}$ is in fact uniformly bounded in $h$ (provided that
$h \leq h_0$) from the finite element convergence theory.  Since the $L^2$
norm is always bounded above by the $H^1$ norm, we have a bound on
$\mcA_h^{-1}$ as a bounded operator on $V_h$ from~\eqref{eq:unifinv}
and hence 
\begin{equation}
  \label{eq:l2invnorm}
  \| A^{-1} M\|_M = \| \mcA_h^{-1} \| \leq C_R(1+CC_Ah_0).
\end{equation}

Since the spectral radius of a matrix is bounded by any natural norm,
we note that~\eqref{eq:minvam} gives an upper bound of
$\mathcal{O}(h^{-2})$ for the largest eigenvalues of $M^{-1}A$,
while~\eqref{eq:l2invnorm} gives an $\mathcal{O}(1)$ lower bound for
the eigenvalues of $A^{-1}M = (M^{-1} A)^{-1}$.

\subsection{The Morse-Ingard system}
\label{sec:femtheory}
Now, we apply this discussion in the context of the Morse-Ingard
system.  
The bilinear form in~\eqref{eq:abstweakform} is continuous on \( H^1 \):
\begin{equation}
  \begin{split}
    |a(U, V)| \leq &  M |\left( \nabla T, \nabla v \right)| + |\left( T, v \right)| \\
    &
    + \tfrac{\gamma-1}{\gamma} |\left( P, v \right)|
    +  \left|\gamma\left( 1 - \tfrac{\Lambda}{M} \right)\right| |\left( T, w \right)| \\
    & + \left( 1 - i \gamma \Lambda \right)
   \left[ \left(\nabla P, \nabla w\right)
     - i \sqrt{\gamma} \langle P , w \rangle_{\Gamma_a} \right] \\
  & - \left[ \gamma \left(1-\tfrac{\Lambda}{M}\right)-
    \tfrac{\Lambda}{M}\right]\left(P, w\right) \\
  \leq & C_c \| U \|_{H^1} \| V \|_{H^1}.
  \end{split}
\end{equation}

If we take $T \equiv 0$ and $P$ with vanishing trace on
$\Gamma_a$, the real part of $a(U, U)$ can
take either positive or negative sign, so $a$ is not coercive.
We can still prove a G{\aa}rding-type inequality for $a$.  Since the solvability
has been proven in~\cite{brennan2015finite}, at least for the
parameter regime of interest, a simple adaptation of the known
techniques for Helmholtz establish discrete solvability and error
estimates for the Morse-Ingard system.
Technically, the resulting theorems postulate a
sufficiently fine mesh, although we have not seen a practical impact
of this for the parameters of interest.  This is consistent with our
earlier discussion that the effective wave numbers are quite moderate.

Let \( U = (T, P) \) and let \( K \) be some real constant, which we
will choose later to ensure that the real part of \( a(U, U) + K \| U \|^2 \)
is bounded below by a constant multiple of \( \| U \|_{H^1}^2 \equiv
\| T \|^2_{H^1} + \| P \|^2_{H^1} \).
We have that
\begin{equation}
  \begin{split}
    a(U, U) = &
    M \| \nabla T \|^2 - i \| T \|^2
    + i \tfrac{\gamma-1}{\gamma} \left( P, T \right)
    + \gamma\left( 1 - \tfrac{\Lambda}{M} \right) \left( T, P \right)
    \\
    & 
    + \left( 1 - i \gamma \Lambda \right)
   \left[ \| \nabla P \|^2
     - i \sqrt{\gamma} \| P \|^2_{\Gamma_a} \right] \\
   &
  - \left[ \gamma \left(1-\tfrac{\Lambda}{M}\right)-
    \tfrac{\Lambda}{M}\right] \| P \|^2 \\
  = &
  M \| \nabla T \|^2 + \| \nabla P \|^2 \\
  & - \left[ \gamma \left(1-\tfrac{\Lambda}{M}\right)  -\tfrac{\Lambda}{M}\right] \| P \|^2
    - \gamma^{\tfrac{3}{2}}\Lambda \| P \|^2_{\Gamma_a} \\
    &  - i \| T \|^2
    - i \gamma \Lambda \| \nabla P \|^2
    - i \sqrt{\gamma} \| P \|^2_{\Gamma_a} \\
&  + i\tfrac{\gamma-1}{\gamma}(P, T)
+ \gamma\left( 1 - \tfrac{\Lambda}{M} \right) \left( T, P \right),
  \end{split}
\end{equation}
where we have separated terms that are purely real or imaginary from
those with indeterminate as to sign.

Considering just the real part, we have
\begin{equation}
  \begin{split}
    \Re\left[a(U, U)\right] = & M \| \nabla T \|^2 + \| \nabla P \|^2
  - \left[ \gamma \left(1-\tfrac{\Lambda}{M}\right)  -\tfrac{\Lambda}{M}\right] \| P \|^2
     \\
    & - \gamma^{\tfrac{3}{2}}\Lambda \| P \|^2_{\Gamma_a} + \Re\left[ i\tfrac{\gamma-1}{\gamma}(P, T)
+ \gamma\left( 1 - \tfrac{\Lambda}{M} \right) \left( T, P \right)\right]
  \end{split}
\end{equation}
By a trace theorem, there exists a constant \( C_{\Gamma_a} \) such that for any \( P \in H^1(\Omega) \),
\begin{equation}
  \| P \|^2_{\Gamma_a} \leq C_{\Gamma_a} \| \nabla P \| \| P \|,
\end{equation}
and so a weighted Young's inequality leads to 
\[
\| P \|^2_{\Gamma_a} \leq \tfrac{1}{2\gamma^{3/2}\Lambda} \|
\nabla P \|^2 + \tfrac{C^2_\Gamma\gamma^{3/2}\Lambda}{2} \| P \|^2.
\]
Since the real parts of the cross terms are bounded below by the
negative of their moduli, we have
\begin{equation}
  \begin{split}
    \Re\left[a(U, U)\right] \geq & M \| \nabla T \|^2 +
    \tfrac{1}{2} \| \nabla P \|^2 \\
    &
    - \left[
      \gamma \left(1-\tfrac{\Lambda}{M}\right)
      -\tfrac{\Lambda}{M}
      + \tfrac{C_{\Gamma_a}^2\gamma^3 \Lambda^2}{2}
      + \tfrac{\gamma-1}{2\gamma} + \tfrac{\gamma}{2}\left|1-\tfrac{\Lambda}{M}\right|
      \right]  \| P \|^2 \\
    & - \left[ \tfrac{\gamma-1}{2\gamma} + \tfrac{\gamma}{2} \left| 1 -
    \tfrac{\Lambda}{M} \right| \right] \| T \|^2.
    \end{split}
\end{equation}

This gives the following G{\aa}rding-type inequality for our bilinear form:
\begin{proposition}
  For any \( K \) with
  \begin{equation}
    K \geq \max\{ \gamma \left(1-\tfrac{\Lambda}{M}\right)
      -\tfrac{\Lambda}{M}
      + \tfrac{C_{\Gamma_a}^2\gamma^3 \Lambda^2}{2}
      + \tfrac{\gamma-1}{2\gamma} +
      \tfrac{\gamma}{2}\left|1-\tfrac{\Lambda}{M}\right|,
      \tfrac{\gamma-1}{2\gamma} + \tfrac{\gamma}{2} \left| 1 -
      \tfrac{\Lambda}{M} \right|\},
  \end{equation}
  there exists an \( \alpha \geq 0 \) such that
  \begin{equation}
    \Re\left[a(U, U)\right] + K \| U \|^2
    \geq \alpha \| U \|_{H^1}^2.
  \end{equation}
\end{proposition}

At this point, the treatment in~\cite{BrennerScott} for general elliptic
problems satisfying such an inequality goes through essentially
unchanged, with straightforward modifications accounting for the
complex-valued nature of the system.

\begin{theorem}
  Suppose the solution to~\eqref{eq:abstweakform} satisfies $\| U
  \|_{H^2} \leq C_R \| F \|$.  Then there exist constants \( h_0 \) and \( C \)
  such that for all meshes with \( h \leq h_0 \), there exists a
  unique Galerkin solution to~\eqref{eq:discabstweakform} such that
  \begin{equation}
    \| U - U_h \|_{H^1} \leq C \inf_{V \in V_H} \| U - V \|_{H^1}
  \end{equation}
  and also
  \begin{equation}
    \| U - U_h \|_{L^2} \leq C_1 C_A C_R h \| U - U_h \|_{H^1}
  \end{equation}
\end{theorem}
This gives a best approximation result in $H^1$ with an extra power of
$h$ in $L^2$.  More precise powers of $h$ can be obtained in terms of postulated regularity.
Also, we note that this theory gives a uniform bound on the discrete
solution operator exactly analogous to~\eqref{eq:unifinv}.

\section{Some block preconditioners and their analysis}
\label{sec:bp}
Having established this convergence theory, we turn to the efficient
solution of the linear systems required to compute the Galerkin approximation.
In particular, we propose and analyze block Jacobi and block
Gauss-Seidel type preconditioners.  Although the formalism is
well-understood for other problems~\cite{mardal2007order,mardal2011preconditioning,wathen1993fast}, and very general implementations are
possible~\cite{Brown:2012}, rigorous analysis requires utilizing the
properties of the particular problem.  

\subsection{Block structure}

The linear system for~\eqref{eq:discabstweakform} naturally leads to a logical block
structure
\begin{equation}
  \label{eq:blockmat}
  \begin{bmatrix} H_T & M_1 \\ M_2 & H_P \end{bmatrix}
  \begin{bmatrix} T \\ P \end{bmatrix}
  =
  \begin{bmatrix} F_1 \\ F_2 \end{bmatrix}.
\end{equation}

We let $M$ denote the mass matrix on
$\MCV_h$ as in~\eqref{eq:massandstiff}.  We also introduce the standard stiffness matrix
\(
K = (\nabla \psi_{\jmath}, \nabla \psi_{\imath}),
\)
which is Hermitian semi-definite.  We
also let $K^\gamma$ denote the stiffness matrix with the transmission
boundary term incorporated:
\[
K^\gamma_{\imath\jmath} = (\nabla \psi_{\jmath}, \nabla \psi_{\imath}) - i\sqrt{\gamma}\langle
\psi_{\jmath} , \psi_{\imath} \rangle_{\Gamma_a}
\]

Comparing to~\eqref{eq:weakform}, we see that $H_T$ is the matrix
arising from the bilinear form
\begin{equation}
\MCM \left( \nabla T, \nabla v \right) - i \left( T, v \right),
\end{equation}
so that
\begin{equation*}
  H_T = \MCM K - i M.
\end{equation*}
We call the associated operator $\mcA_{T,h}:\MCV_h \rightarrow \MCV_h^\prime$. 
The bilinear form for $H_T$ has a semidefinite real part and satisfies
a G{\aa}rding inequality for any $K > 0$.  In fact, a compactness
argument can be used to establish coercivity.  It is uniformly
solvable for all $h$, and the inverse of $\mcA_{T,h}$ is uniformly
bounded (in both $H^1$ and $L^2$ norms) with
\begin{equation}
  \label{eq:HTinv}
  \| \mcA^{-1}_{T,h} \| \leq C_T.
\end{equation}

The matrix $H_P$ arises from the bilinear form
\[
\left( 1 - i \gamma \Lambda \right)
\left[ \left(\nabla P, \nabla w\right)
  - i \sqrt{\gamma} \langle P , w \rangle_{\Gamma_a} \right]
- \left[ \gamma \left(1-\tfrac{\Lambda}{\MCM}\right)-
  \tfrac{\Lambda}{\MCM}\right]\left(P, w\right),
\]
and we can write the matrix $H_P$ as
\[
H_P = 
(1-i \gamma \Lambda) K^\gamma - \left[ \gamma(1-\tfrac{\Lambda}{\MCM})
  - \tfrac{\Lambda}{\MCM} \right] M
\]

The invertibility of the underlying differential operator can be
established by standard means, and discrete solvability and error
estimates via a G{\aa}rding inequality techniques as
in~\cite{BrennerScott} as for the overall system.  Identifying the underlying operator as $\mcA_{P,h}:\MCV_h \rightarrow \MCV_h^\prime$, we have a uniform bound
\begin{equation}
  \label{eq:HPinv}
  \| \mcA^{-1}_{P,h} \| \leq C_P,
\end{equation}
at least for $h \leq h_0$.

The off-diagonal blocks $M_{\imath}$ in~\eqref{eq:blockmat} are both scaled mass matrices, with
\begin{equation}
  \label{eq:Mi}
M_1 = i \tfrac{\gamma-1}{\gamma} M, \ \ \
M_2 = \gamma\left(1-\tfrac{\Lambda}{\MCM} \right) M.
\end{equation}

\subsection{Block preconditioners}
\label{subsec:bp}
As the matrices in~\eqref{eq:blockmat} are large, sparse, 
and ill-conditioned, their solution at scale will require effective
preconditioners.  
We consider two families of block preconditioners in this analysis.
A block Jacobi preconditioner simply consists of the diagonal blocks, with
\begin{equation}
  \label{eq:PJ}
  \Pi_{J} = \begin{bmatrix} H_T & 0 \\ 0 & H_P \end{bmatrix}.
\end{equation}
The block Gauss-Seidel preconditioner consists of the lower triangle
of blocks, with
\begin{equation}
  \label{eq:GS}
  \Pi_{GS} = \begin{bmatrix} H_T & 0 \\ M_2 & H_P \end{bmatrix}.
\end{equation}

At each iteration of a Krylov method using one of these
preconditioners, we must invert the $\Pi$ onto the current residual.
Doing so requires 
inverting both of the diagonal blocks $H_T$ and $H_P$.  Inverting $\Pi_{GS}$ also requires a
matrix-vector product with $M_2$ and some vector arithmetic, although
these operations are typically much cheaper than solving linear
systems with the diagonal blocks.

The block matrices $\Pi_J$ and $\Pi_{GS}$ arise from dropping terms
from the bilinear form $a$ in~\eqref{eq:aandF}.  That is, again with
$U=(T,P)$ and $V=(v,w)$, we define
\begin{equation}
  \begin{split}
  a_J(U, V) = & \MCM \left( \nabla T, \nabla v \right) - i \left( T, v
  \right) \\
    &  + \left( 1 - i \gamma \Lambda \right)
    \left[ \left(\nabla P, \nabla w\right)
      - i \sqrt{\gamma} \langle P , w \rangle_{\Gamma_a} \right] \\
    & - \left[ \gamma \left(1-\tfrac{\Lambda}{\MCM}\right)-
      \tfrac{\Lambda}{\MCM}\right]\left(P, w\right)
  \end{split}
  \label{eq:aJ}
\end{equation}
\begin{equation}
  \begin{split}
    a_{GS}(U, V) = & \MCM \left( \nabla T, \nabla v \right) - i \left(
    T, v \right) \\
    & +\gamma\left( 1 - \tfrac{\Lambda}{\MCM} \right) \left( T, w
    \right) \\
    & 
    + \left( 1 - i \gamma \Lambda \right)
    \left[ \left(\nabla P, \nabla w\right)
      - i \sqrt{\gamma} \langle P , w \rangle_{\Gamma_a} \right] \\
    & - \left[ \gamma \left(1-\tfrac{\Lambda}{\MCM}\right)-
       \tfrac{\Lambda}{\MCM}\right]\left(P, w\right) \\
  \end{split}
  \label{eq:aGS}
\end{equation}
The matrices $\Pi_J$ and $\Pi_{GS}$ are obtained by applying each of
these bilinear forms to pairs of finite element basis functions in the
natural way.  Following a similar discussion as the bilinear form $a$,
it is not hard to see that both of these forms are continuous in
$(H^1)^2$ and satisfy a G{\aa}rding-type inequality with constant $K$
no worse than that for the full Morse-Ingard system~\eqref{eq:aandF}.
Consequently, the inverse operators for these variational problems
will have mesh-independent bounds (for $h \leq h_0$) of their inverses.

For each preconditioner $\Pi=\Pi_J, \Pi_{GS}$, we wish to assess the eigenvalues of
$\Pi^{-1} A$, where $A$ is the matrix in~\eqref{eq:blockmat}.
Equivalently, we can study the generalized eigenvalues of $A$ with
respect to $\Pi$. If these eigenvalues exhibit favorable properties,
such as mesh-independent clustering away from 0, then we can hope for
scalable solution of the system by means of a Krylov method.

The generalized eigenvalue problem for the block-Jacobi preconditioner is 
\begin{align}
  \label{eq:bjpgep}
\begin{bmatrix}
H_T & M_1 \\
M_2 & H_P
\end{bmatrix} 
\begin{bmatrix}
T\\
P
\end{bmatrix} &= \lambda \begin{bmatrix} H_T& 0\\
0 & H_P
\end{bmatrix}
\begin{bmatrix}
T\\
P
\end{bmatrix}.
\end{align}

Writing this out as a system of equations gives

\begin{align*} H_TT + M_1 P &= \lambda H_T T\\
M_2T + H_P P &= \lambda H_P P.
\end{align*}

First, we can rule out the possibility of an eigenvalue $\lambda =1$ because $M_i$ are both nonsingular.  With $\lambda \neq 1$, we can eliminate $P$ from the second equation and insert into the first to find that

%
%
%

\begin{align*} 
H_T^{-1}M_1 H_P^{-1}M_2T   &=( \lambda -1)^2 T,
\end{align*}
from which we have:
\begin{proposition}
The generalized eigenvalues of~\eqref{eq:bjpgep} satisfy
$(\lambda-1)^2 = \mu$ for any eigenvalue $\mu$ of $H_T^{-1}M_1
H_P^{-1}M_2$, and we have $\lambda = 1 \pm \sqrt{\mu}$.
\end{proposition}

Similarly, the generalized eigenvalue problem for the block lower Gauss-Seidel preconditioner is
\begin{align}
  \label{eq:bgsgep}
\begin{bmatrix}
H_T & M_1 \\
M_2 & H_P
\end{bmatrix} 
\begin{bmatrix}
T\\
P
\end{bmatrix} &= \lambda \begin{bmatrix} H_T& 0\\
M_2 & H_P
\end{bmatrix}
\begin{bmatrix}
T\\
P
\end{bmatrix}.
\end{align}
This translates to the system of equations
\begin{align*} H_TT + M_1 P &= \lambda H_T T\\
  M_2T + H_P P &= \lambda M_2 T+ \lambda H_P P.
\end{align*}

Unlike the block-Jacobi case, $\lambda=1$ is a real possibility for
the block Gauss-Seidel case.  To see this, if $\lambda =1$,
the first equation with implies that
$P=0$, while the second is trivial.  Consequently,
any nonzero $\begin{bmatrix} 0 \\ T \end{bmatrix}$ is an eigenvector
corresponding to $\lambda = 1$.  Some basic algebra then gives
\begin{proposition}
  The generalized eigenvalues for~\eqref{eq:bgsgep} are $\lambda =1$
  with multiplicity $\dim \MCV_h$ and $1-\mu$ for any eigenvalue $\mu$
  of $H_T^{-1}M_1H_P^{-1} M_2$.
\end{proposition}

From these two propositions, it is clear that understanding the
eigenvalues of the matrix $H_T^{-1}M_1H_P^{-1} M_2$ is critical in
evaluating either preconditioner.

Using~\eqref{eq:Mi}, we write $H_T^{-1} M_1 = \sigma_1 H_T^{-1} M$
with
$\sigma_1 = i\tfrac{\gamma-1}{\gamma}$ 
and
$H_P^{-1} M_2 = \sigma_2 H_P^{-1} M$ with $\sigma_2 = \gamma
\left(1-\tfrac{\Lambda}{\MCM}\right)$.  We can
bound the eigenvalues of $H_T^{-1} M H_P^{-1} M$ and scale the result
by the $|\sigma_1\sigma_2|$.

\begin{lemma}
  The eigenvalues $\lambda_T$ of $H_T^{-1}M$ and $\lambda_P$ of $H_P^{-1} M$ are bounded by
  \begin{equation}
    |\lambda_T| \leq C_T, \ \ \ |\lambda_P| \leq C_P.
  \end{equation}
\end{lemma}
\begin{proof}
  The spectral radius of a matrix is bounded by any
  natural matrix norm and hence that induced by the $M$-norm.
  Applying~\eqref{eq:l2invnorm} with the particular bounds
  given in~\eqref{eq:HTinv} and~\eqref{eq:HPinv} gives the desired result.
\end{proof}

\begin{proposition}
  The eigenvalues $\mu$ of $H_T^{-1}M_1H_P^{-1} M_2$ are bounded by
  \begin{equation}
    |\mu| \leq C_T C_P \left(\gamma-1 \right) \left( 1-
    \tfrac{\Lambda}{\MCM} \right).
  \end{equation}
\end{proposition}

\begin{proof}
  We bound the spectral radius of $H_T^{-1}M_1H_P^{-1} M_2$ by its
  $M$-norm and use the previous lemmas and the values for $\sigma_1$
  and $\sigma_2$ so that
  \begin{equation*}
  \begin{split}
    |\mu| & \leq \|H_T^{-1}M_1H_P^{-1} M_2\|_M \\
    &\leq
  \| H_T^{-1}M_1 \|_M \|H_P^{-1} M_2 \|_M \\
  & \leq |\sigma_1|  \| H_T^{-1}M \|_M |\sigma_2| \|H_P^{-1} M \|_M \\
 &  \leq \left( \tfrac{\gamma-1}{\gamma} C_T \right)
  \left( \gamma \left(1-\tfrac{\Lambda}{\MCM}\right) C_P \right) \\
  & = C_T C_P \left(\gamma-1 \right) \left( 1-\tfrac{\Lambda}{\MCM}
  \right).
  \end{split}
  \end{equation*}
\end{proof}

We now have mesh-independent upper bounds of the eigenvalues of our
preconditioned linear systems.
\begin{theorem}
  \label{thm:j}
  The eigenvalues of~\eqref{eq:bjpgep} are contained in a ball of
  radius $\sqrt{C_T C_P \left(\gamma-1 \right) \left(
    1-\tfrac{\Lambda}{\MCM}\right)}$ around 1.
\end{theorem}
\begin{theorem}
  \label{thm:gs}
  The generalized eigenvalue problem~\eqref{eq:bgsgep} has an
  eigenvalue 1 with multiplicity $\dim \MCV_h$, and
  another $\dim \MCV_h$ contained in a ball of radius
  $C_T C_P \left(\gamma-1 \right) \left(1-\tfrac{\Lambda}{\MCM}\right)$ around 1.
\end{theorem}

Although the bounds on the block Jacobi preconditioner appear tighter,
the high multiplicity of the eigenvalue 1 for the block Gauss-Seidel
preconditioner is very powerful.  Our numerical results will indicate
that the Gauss-Seidel preconditioner is indeed preferable in practice.

While these upper bounds are in fact mesh-independent, they do not
preclude the ball around 1 containing the origin.  However, we can
rule out 0 both as an eigenvalue and an accumulation point of the
eigenvalues of the preconditioned system.  If we equip $\MCV_h$ with
the $H^1$ topology, $\Pi^{-1} A$ discretizes a bounded operator with
bounded inverse from $\MCV_h$ into itself.  Hence, the eigenvalues
must be bounded away from zero.  Although we could give a finer
estimate by adapting our recent discussion for $A^{-1} \Pi$ instead of
$\Pi^{-1} A$, it is sufficient to realize that the operator
$\mathcal{A}_h$ associated with the bilinear form $a$ has a uniformly
bounded inverse and that the preconditioning bilinear forms are both
bounded operators.

\section{Numerical results}
\label{sec:num}
Our numerical results focus on assessing the performance of the
preconditioners we have developed and analyzed above.  Our numerical
tests are performed using a development branch of Firedrake~\cite{Rathgeber:2016}
that uses complex arithmetic.  Our unstructured meshes of the tuning
fork exterior are generated using
\texttt{gmsh}~\cite{geuzaine2009gmsh}.  We have tested our techniques
in a three-dimensional setting as well, \footnote{Thanks to Artur
  Safin for providing the \texttt{gmsh} input for these cases.} with
similar results but slightly higher iteration counts for the
preconditioners.  We present only experiments for the parameter values
of the QEPAS configuration.  The ROTADE parameters lead to quite
similar results.

\begin{figure}
  \begin{centering}
    \include{tikzmesh}
  \end{centering}
  \caption{Coarsest two-dimensional mesh of the two-dimensional tuning
    fork exterior consisting of 695 triangles and 423 vertices.  For
    this domain, $\Gamma_a$ is the outer rectangle, and $\Gamma_w$ is
    the inner boundary of the tuning fork. }
  \label{fig:tfmesh}    
\end{figure}

\begin{figure}
  \centering
  \begin{subfigure}[t]{0.475\textwidth}
    \centering
    \includegraphics[width=0.9\textwidth]{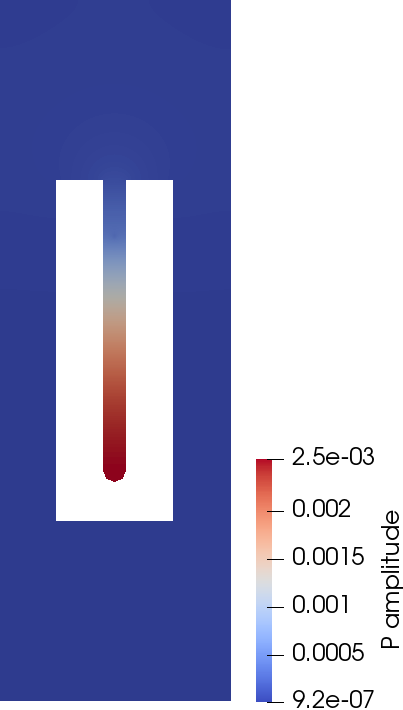}
    \caption{Pressure amplitude}
    \label{figP}
  \end{subfigure}
  \begin{subfigure}[t]{0.475\textwidth}
    \centering
    \includegraphics[width=0.9\textwidth]{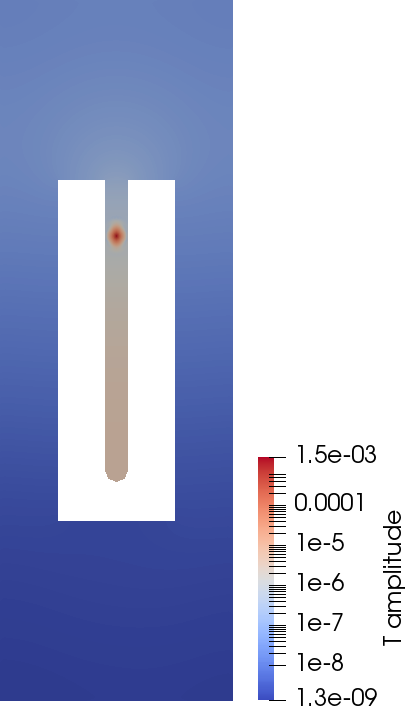}
    \caption{Temperature amplitude}
    \label{figT}
  \end{subfigure}
  \caption{Computed amplitude of pressure and temperature due to
    source term between the tuning fork tines.  Because of the rapid
    temperature decay away from the laser source, we have plotted the
    temperature amplitude on a log scale.}
  \label{fig:tfPT}
  \end{figure}

First, we numerically computed the eigenvalues of the original and
preconditioned systems on the coarse mesh shown in Figure~\ref{fig:tfmesh}.  This computation uses LAPACK
and is quite expensive so can only be performed on a very coarse mesh.
We plot the eigenvalues for the unpreconditioned system in
Figure~\ref{nopc} and those using both block preconditioners in
Figure~\ref{pc}.
The former plot illustrates two features.  First, it clearly
highlights the nature of operators being  coupled -- the temperature
satisfies a perturbation of the heat equation while the pressure a
perturbation of the wave equation.  These correspond to the two
(approximate) lines of eigenvalues extending along the imaginary and
real axis, respectively.  Due to the coupling, the eigenvalues are
slightly off of the respective lines.  Eigenvalues approaching zero
clearly indicate the need for preconditioning, although it appears
that preconditioning strategies that effectively capture the behavior
of the heat and wave equations should be highly effective for the
coupled problem.

The latter plot, showing the eigenvalues of the preconditioned systems, illustrates Theorems~\ref{thm:j}
and~\ref{thm:gs} above.  In practice, all the preconditioned
eigenvalues lie reasonably near one.   The clustering is a bit tighter
in practice for Gauss-Seidel than for Jacobi.  Moreover, fully half of
the eigenvalues using Gauss-Seidel are 1 to machine precision. We similar behavior for a very coarse three-dimensional mesh. 

\begin{figure}
  \centering
  \begin{subfigure}[t]{0.4\textwidth}
    \centering
      \begin{tikzpicture}[scale=0.7]
        \begin{axis}[
            xlabel={$\Re(\lambda)$},
            ylabel={$\Im(\lambda)$},
            title={Unpreconditioned Spectrum}
          ]
          \addplot[mark=+, only marks] table
                  [x=Real, y=Imag, col sep=comma] {nopc.eig.QEPAS.2d.695.csv};
        \end{axis}
      \end{tikzpicture}
      \caption{Eigenvalues of the unpreconditioned system, many of which come close to the origin.}
      \label{nopc}
  \end{subfigure}%
  \hspace{1.0in}
  \begin{subfigure}[t]{0.4\textwidth}
    \centering
      \begin{tikzpicture}[scale=0.7]
        \begin{axis}[
            xmin=0, xmax=2,
            xlabel={$\Re(\lambda)$},
            ylabel={$\Im(\lambda)$},
            title={Preconditioned Spectrum}
          ]
          \addplot[red, mark=o, only marks] table
                  [x=Real, y=Imag, col sep=comma]
                  {gs.eig.QEPAS.2d.695.csv};
                  \addlegendentry{GS}
          \addplot[blue, mark=x, only marks] table
                  [x=Real, y=Imag, col sep=comma]
                  {jac.eig.QEPAS.2d.695.csv};
                  \addlegendentry{Jacobi}
        \end{axis}
      \end{tikzpicture}
      \caption{Eigenvalues obtained with the block Gauss-Seidel (GS) and
        Jacobi (Jacobi) preconditioners, both of which cluster near 1 and away from the origin.}
  \label{pc}
  \end{subfigure}%
  \caption{Eigenvalues with and without preconditioning for the
    two-dimensional Morse-Ingard system on the mesh in
    Figure~\ref{fig:tfmesh} using QEPAS parameters.}
  \label{fig:pceig}
\end{figure}
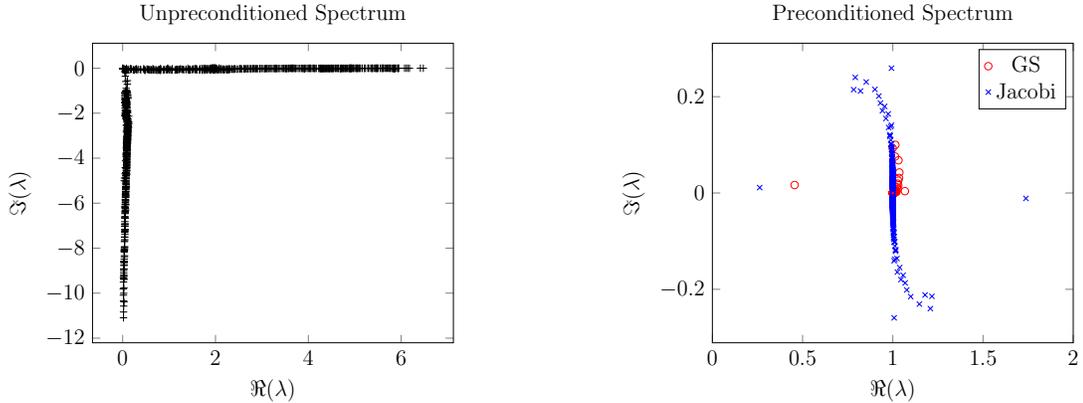

Based on the eigenvalue clustering in Figures~\ref{fig:pceig},
we expect
very low GMRES iteration counts using these
preconditioners.  We studied the scaling of various
solvers under a sequence of uniform mesh refinements.  The coarsest
two-dimensional mesh has but 423 vertices, while the finest has 
about 358k vertices.

PETSc~\cite{Balay:2016, Balay:1997} provides Firedrake's numerical
linear algebra capability.  For testing the block preconditioners, we
assemble the system stiffness matrix into a 
\texttt{MatNest} segregating rows and columns corresponding to
pressure and temperature and then make use of \texttt{FieldSplit}
~\cite{Brown:2012} with a Krylov method.  In this way, switching between
Gauss-Seidel and block Jacobi preconditioners amounts to selecting
either a \texttt{multiplicative} or \texttt{additive} option for the
preconditioner.

To establish a baseline for the iteration count for the two block preconditioners, we first consider inverting those blocks with a sparse LU factorization.
So, we use GMRES for the outer solver with a relative tolerance of
$10^{-8}$ on the outer solves, using the default PETSc LU
factorization for inverting the diagonal blocks.  These results correspond to the labels `GS' and `Jac' in Figure~\ref{fig:itsandtime}.

It is also possible to invert these diagonal blocks by a
preconditioned iterative method, using flexible
GMRES~\cite{saad1993flexible} instead of standard GMRES for the outer
iteration.  Rather than using the resulting nested iteration, we make a
further approximation by replacing the inner solve with a
preconditioner.  We use \texttt{gamg}~\cite{adams2004algebraic}
algebraic multigrid with options 
\texttt{-pc\_mg\_type full -mg\_levels\_pc\_type sor -mg\_levels\_ksp\_type
  gmres -mg\_levels\_ksp\_max\_it 50} within the appropriate prefix for
each diagonal block.
These results appear under labels `GS/GAMG' and `Jac/GAMG' in figure~\ref{fig:itsandtime}.
We remark that using a Krylov method on subgrids
for Helmholtz was first proposed in~\cite{elman2001multigrid}.

We see an essentially mesh-independent iteration count in Figure~\ref{it2d} for
our all of our block preconditioners, although the block Gauss-Seidel
preconditioner takes roughly half as many outer iterations as the
block Jacobi one.  Somewhat surprisingly, replacing the inner solve
with an algebraic multigrid preconditioner seems to \emph{reduce} the
outer iteration count slightly for block Jacobi.

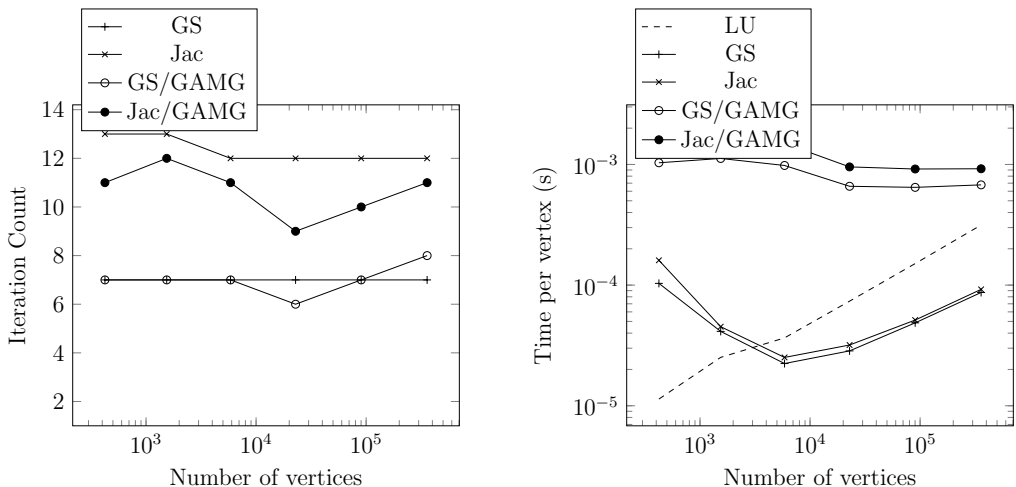
\begin{figure}
  \centering
  \begin{subfigure}[t]{0.49\textwidth}
    \begin{tikzpicture}[scale=0.75]
      \begin{axis}[xmode=log, ymin =1, xlabel={Number of vertices},
          ylabel={Iteration Count}, legend style={at={(0.25,1.3)}, anchor=north}
          ]
        \addplot[mark=+] table
                [x=N, y=Its, col sep=comma]{pc.Safin.multlu.csv};
        \addlegendentry{GS}
        \addplot[mark=x] table
                [x=N, y=Its, col sep=comma]
                {pc.Safin.addlu.csv};
        \addlegendentry{Jac}
        \addplot[mark=o] table
                [x=N, y=Its, col sep=comma]{pc.Safin.multgamg.csv};
        \addlegendentry{GS/GAMG}
        \addplot[mark=*] table
                [x=N, y=Its, col sep=comma]{pc.Safin.addgamg.csv};
        \addlegendentry{Jac/GAMG}
      \end{axis}
    \end{tikzpicture}
    \caption{Iteration counts for various block preconditioning strategies.}
    \label{it2d}
  \end{subfigure}
  \begin{subfigure}[t]{0.49\textwidth}
    \centering
    \begin{tikzpicture}[scale=0.75]
      \begin{axis}[xmode=log, ymode=log, xlabel={Number of vertices},
          ylabel={Time per vertex (s)},
          legend style={at={(0.25,1.3)}, anchor=north}]
        \addplot[dashed] table
                [x=N, col sep=comma, y expr=\thisrowno{1}/\thisrowno{0}]{pc.Safin.lu.csv};
                \addlegendentry{LU}
        \addplot[mark=+] table
                [x=N, col sep=comma, y expr=\thisrowno{1}/\thisrowno{0}]{pc.Safin.multlu.csv};
                \addlegendentry{GS}
        \addplot[mark=x] table
                [x=N, col sep=comma, y expr=\thisrowno{1}/\thisrowno{0}]{pc.Safin.addlu.csv};
                \addlegendentry{Jac}
        \addplot[mark=o] table
                [x=N, col sep=comma, y expr=\thisrowno{1}/\thisrowno{0}]{pc.Safin.multgamg.csv};
                \addlegendentry{GS/GAMG}
        \addplot[mark=*] table
                [x=N, col sep=comma, y expr=\thisrowno{1}/\thisrowno{0}]{pc.Safin.addgamg.csv};
                \addlegendentry{Jac/GAMG}      
      \end{axis}
    \end{tikzpicture}
    \caption{Solver timings for
      various block preconditioning strategies.}
    \label{time2d}
  \end{subfigure}
\caption{Performance of block preconditioners for the
  two-dimensional Morse-Ingard equations.
  Preconditioners `GS' and `Jac' correspond
  to inverting the diagonal blocks $H_T$ and $H_P$ with
  LU factorization, and `GS/GAMG' and `Jac/GAMG'
  correspond to replacing the inversion of $H_T$ and $H_P$ with a
  single sweep of \texttt{gamg}.  A sparse LU factorization is included for comparison.}
\label{fig:itsandtime}
\end{figure}

We also measure the actual run-time of our various solver
options.  We report the solver time (including preconditioner
setup but not stiffness matrix assembly)
in Figure~\ref{time2d}.  We report the time normalized by the number
of vertices, which is half the total number of unknowns, versus the
total number of vertices.  In this metric, a flat curve corresponds to
exact linear scaling.  We see an uptick for the block Jacobi and
Gauss-Seidel preconditioners using LU factorization for the blocks
(this is expected from the superlinear complexity of Gaussian
elimination).  This trend will continue under mesh refinement until factoring
the sub-blocks becomes uncompetitive, but we expect the flat trend to
continue for AMG.

An important point regarding the reformulated system appears in
Figure~\ref{fig:itsandtimeBK},
where we apply the same kinds of block
preconditioners to the form of the Morse-Ingard equations considered
in~\cite{brennan2015finite}.  When the diagonal blocks are inverted
with sparse direct factorization, we see flat iteration counts for
both preconditioners.  The block Gauss-Seidel preconditioner requires 1-2
more iterations than for the form we consider in this paper, and the
block Jacobi preconditioner requires 3-5 more iterations.  However,
apparently the subblocks for the formulation
in~\cite{brennan2015finite} are less amenable to preconditioning than
those we consider here.  Replacing the inversion of diagonal blocks
with an application of \texttt{gamg} usually leads to somewhat higher
(about a factor of two) iteration counts and run-times, but on certain
meshes the iteration count either spiked wildly or GMRES failed to
converge altogether.  Apparently, the manipulation leading
from~\eqref{eq:dropstar} to \eqref{eq:PDE} reduces the effective wave
number or otherwise improves the conditioning of the diagonal blocks.

\begin{figure}
  \centering
  \begin{subfigure}[t]{0.49\textwidth}
    \begin{tikzpicture}[scale=0.75]
      \begin{axis}[xmode=log, ymode=log, ymin =1, xlabel={Number of vertices},
          ylabel={Iteration Count}, legend style={at={(0.25,1.3)}, anchor=north}
          ]
        \addplot[mark=+] table
                [x=N, y=Its, col sep=comma]{pc.BrennanKirby.multlu.csv};
        \addlegendentry{GS}
        \addplot[mark=x] table
                [x=N, y=Its, col sep=comma]
                {pc.BrennanKirby.addlu.csv};
        \addlegendentry{Jac}
        \addplot[mark=o] table
                [x=N, y=Its, col sep=comma]{pc.BrennanKirby.multgamg.csv};
        \addlegendentry{GS/GAMG}
        \addplot[mark=*] table
                [x=N, y=Its, col sep=comma]{pc.BrennanKirby.addgamg.csv};
        \addlegendentry{Jac/GAMG}
      \end{axis}
    \end{tikzpicture}
    \caption{Iteration counts for various block preconditioning strategies.}
    \label{it2dBK}
  \end{subfigure}
  \begin{subfigure}[t]{0.49\textwidth}
    \centering
    \begin{tikzpicture}[scale=0.75]
      \begin{axis}[xmode=log, ymode=log, xlabel={Number of vertices},
          ylabel={Time per vertex (s)},
          legend style={at={(0.25,1.3)}, anchor=north}]
        \addplot[mark=+] table
                [x=N, col sep=comma, y expr=\thisrowno{1}/\thisrowno{0}]{pc.BrennanKirby.multlu.csv};
                \addlegendentry{GS}
        \addplot[mark=x] table
                [x=N, col sep=comma, y expr=\thisrowno{1}/\thisrowno{0}]{pc.BrennanKirby.addlu.csv};
                \addlegendentry{Jac}
        \addplot[mark=o] table
                [x=N, col sep=comma, y expr=\thisrowno{1}/\thisrowno{0}]{pc.BrennanKirby.multgamg.csv};
                \addlegendentry{GS/GAMG}
        \addplot[mark=*] table
                [x=N, col sep=comma, y expr=\thisrowno{1}/\thisrowno{0}]{pc.BrennanKirby.addgamg.csv};
                \addlegendentry{Jac/GAMG}      
      \end{axis}
    \end{tikzpicture}
    \caption{Solver timings for
      various block preconditioning strategies.}
    \label{time2dBK}
  \end{subfigure}
\caption{Performance of various block preconditioners for the
  two-dimensional Morse-Ingard equations  as formulated
  in~\cite{brennan2015finite}.  Preconditioners `GS' and `Jac' correspond
  to inverting the diagonal blocks with LU factorization,
  while `GS/GAMG' and `Jac/GAMG' correspond to replacing each inversion
  of $H_T$ and $H_P$ with a single sweep of \texttt{gamg}.}
\label{fig:itsandtimeBK}
\end{figure}
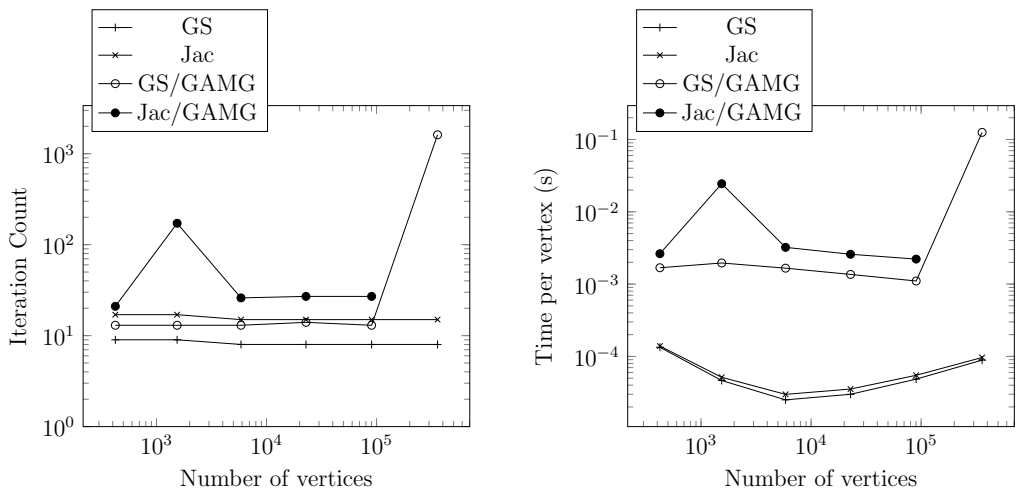

Seeing that our block preconditioners with inner solves approximated
by \texttt{gamg} are very favorable, we have tested these
configurations on a three-dimensional version of this problem.  
Sparse direct solvers are far less competitive in three dimensions, so
we have not considered these options.

Our coarsest mesh contains 209 vertices, the finest about
62,000.  The iteration counts are slightly higher than for two
dimensions and seems to trend upward slightly under mesh refinement.
Iteration counts and timings are shown in Figure~\ref{fig:itsandtime3d}.

\begin{figure}
  \centering
  \begin{subfigure}[t]{0.49\textwidth}
    \begin{tikzpicture}[scale=0.75]
      \begin{axis}[xmode=log, ymin =1, xlabel={Number of vertices},
          ylabel={Iteration Count}, legend style={at={(0.25,1.3)}, anchor=north}
          ]
        \addplot[mark=o] table
                [x=Verts, y=Its, col sep=comma]{solving.QEPAS.multgamg.3d.csv};
        \addlegendentry{GS/GAMG}
        \addplot[mark=*] table
                [x=Verts, y=Its, col sep=comma]{solving.QEPAS.addgamg.3d.csv};
        \addlegendentry{Jac/GAMG}
      \end{axis}
    \end{tikzpicture}
    \caption{Iteration counts for various block preconditioning strategies.}
    \label{it3d}
  \end{subfigure}
  \begin{subfigure}[t]{0.49\textwidth}
    \centering
    \begin{tikzpicture}[scale=0.75]
      \begin{axis}[xmode=log, ymode=log, ymin=0.001, ymax=0.01, xlabel={Number of vertices},
          ylabel={Time per vertex (s)},
          legend style={at={(0.25,1.3)}, anchor=north}]
        \addplot[mark=o] table
                [x=Verts, col sep=comma, y expr=\thisrowno{2}/\thisrowno{0}]{solving.QEPAS.multgamg.3d.csv};
        \addlegendentry{GS/GAMG}
        \addplot[mark=*] table
                [x=Verts, col sep=comma, y expr=\thisrowno{2}/\thisrowno{0}]{solving.QEPAS.addgamg.3d.csv};
        \addlegendentry{Jac/GAMG}      
      \end{axis}
    \end{tikzpicture}
    \caption{Solver timings (not including matrix assembly) for
      various block preconditioning strategies.}
    \label{time3d}
  \end{subfigure} 
\caption{Performance of block Gauss-Seidel and Jacobi preconditioners
  for the three-dimensional Morse-Ingard equations.  In both cases,
  inverses of the diagonal blocks $H_T$ and $H_P$ are approximated
  with a single application of \texttt{gamg}.  Iteration counts are on
  the left, and timings are on the right.}
\label{fig:itsandtime3d}
\end{figure}
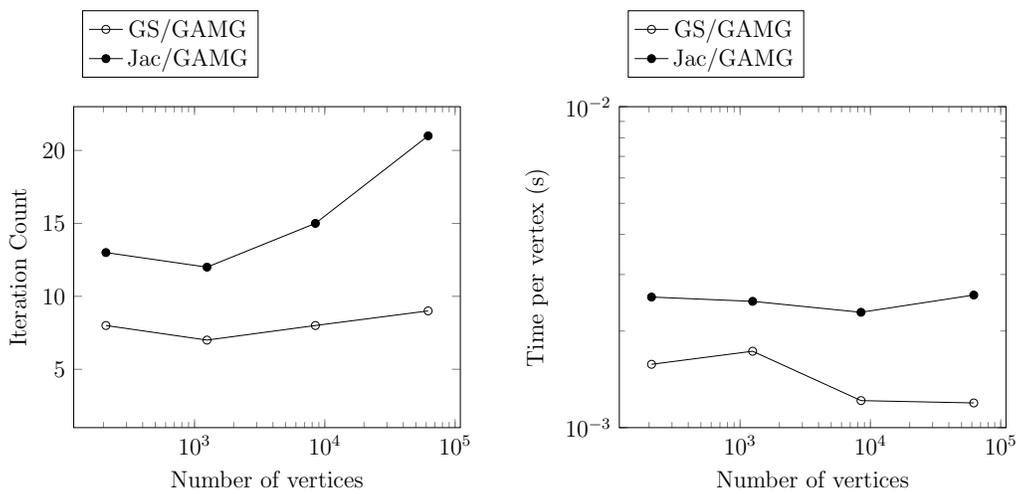

\section{Conclusions}
We have developed finite element theory for a particular form of the
Morse-Ingard equations for thermoacoustic systems.  Although not
coercive, this form has a looser coupling between pressure and
temperature than the form considered in~\cite{brennan2015finite}.  Discrete solvability and error estimates follow from an
adaptation of G{\aa}rding's inequality.  We also study
highly-effective block preconditioners for linear systems.  In
addition to satisfying rigorous eigenvalue bounds, the preconditioners
deliver highly practical iteration counts.

In the future, we hope to study more advanced physical
configurations that include coupling of the pressure/temperature
equations both to thermal and vibrational effects in the tuning fork
and to the equations of fluid motion.  This will be a highly
nontrivial multiphysics setting that requires effective solvers.



\bibliographystyle{elsarticle-num} 
\bibliography{myBib}





\end{document}